\documentclass[a4paper,10pt]{article}
\usepackage[latin1]{inputenc}
 \usepackage[T1]{fontenc}
  \usepackage{amsmath,amsthm,amssymb}
 \usepackage[all]{xy}

\def \N{{\mathbb N}}

\def \R{{\mathbb R}}

\def \1{{\mathbb 1}}

\newtheorem{Exemp}{Examples}
\newtheorem{Thm}{Theorem}
\newtheorem{Prop}{Proposition}

\newtheorem{Lem}{Lemma}

\newtheorem{Def Nota}{Definitions and notations} 
\newtheorem{Cor}{Corollary}

\font\ninerm=cmr9
\long\outer\def\abstract#1{\bigskip\vbox{\noindent\ninerm
\baselineskip=10pt#1}\nobreak\bigskip}
\newcount\numero
\def\exo#1{\advance\numero by 1\bigskip
{\noindent\tenbf #1\the\numero. }}
\def\frac#1#2{{#1\over #2}}
\title{Limited operators and differentiability.}   
\author{Mohammed Bachir}
\begin{document}
\maketitle
\begin{center} {\it Laboratoire SAMM 4543, Universit\'e Paris 1 Panth\'eon-Sorbonne, Centre P.M.F. 90 rue Tolbiac 75634 Paris cedex 13}
\end{center}
\begin{center} 
{\it Email : Mohammed.Bachir@univ-paris1.fr}
\end{center}
\noindent\textbf{Abstract.} We characterize the limited operators by differentiability of convex continuous functions. Given Banach spaces $Y$ and $X$ and a linear continuous operator $T: Y \longrightarrow X$, we prove that   $T$ is a limited operator if and only if, for every convex continuous function $f: X \longrightarrow \R$ and every point $y\in Y$, $f\circ T$ is Fr\'echet differentiable at $y\in Y$ whenever $f$ is G\^ateaux differentiable at $T(y)\in X$. 

\vskip5mm
\noindent {\bf Keyword, phrase:} Limited operators, G\^ateaux differentiability, Fr\'echet differentiability, convex functions, extreme points.\\
{\bf 2010 Mathematics Subject:} Primary 46G05, 49J50, 58C20, 46B20, secondary 47B07.

\begin{center}
 \section{Introduction.}
\end{center} 

A subset $A$ of a Banach space $X$ is called limited, if every weak$^*$ null sequence $(p_n)_n$ in $X^*$ converges uniformly on $A$, that is,$$\lim_{n\rightarrow +\infty} \sup_{x\in A}|\langle p_n , x \rangle|=0.$$
We know that every relatively compact subset of $X$ is limited, but the converse is false in general. A bounded linear operator $T : Y \longrightarrow X$ between Banach spaces $Y$ and $X$ is called limited, if $T$ takes the closed unit ball $B_Y$ of $Y$ to a limited subset of $X$. It is easy to see that $T : Y \longrightarrow X$ is limited if and only if, the adjoint operator $T^*: X^*\longrightarrow Y^*$ takes weak$^*$ null sequence to norm null sequence. For useful properties of limited sets and limited operators we refer to \cite{Sc}, \cite{BD}, \cite{CGL} and \cite{A}. 

We know that in a finite dimensional Banach space, the notions of G\^ateaux and Fr\'echet differentiability coincide for convex continuous functions. In \cite{BF}, Borwein and Fabian proved that a Banach space $Y$ is infinite dimensional if and only if, there exists on $Y$ functions $f$ having points at
which $f$ is G\^ateaux but not Fr\'echet differentiable. They also pointed in the introduction of \cite{BF} the observation that if the sup-norm $\|.\|_{\infty}$ on $c_0$ is G\^ateaux differentiable at some point, then it is Fr\'echet differentiable there. In this article we observe that this phenomenon is not just related to the sup-norm but more generally, for each convex lower semicontinuous function $g: l^{\infty} \rightarrow \R \cup \lbrace +\infty\rbrace$, if $g$ is G\^ateaux differentiable at some point $a\in c_0$ which is in the interior of its domain, then the restriction of $g$ to $c_0$ is Fr\'echet differentiable at $a$. This hold in particular when $g=(f^*)^*$ is the Fenchel biconjugate of a convex continuous function $f:c_0 \rightarrow \R$.
In fact, this phenomenon  is due, (see Corollary \ref{Za1} in the Appendix and the comment just before), to the fact that the canonical embedding $i: c_0\longrightarrow l^{\infty}$ is a limited operator (see the reference \cite{CGL}). 

\vskip5mm

The goal of this paper, is to prove the following characterization of limited operators in terms of the coincidence of G\^ateaux and Fr\'echet differentiability of convex continuous functions. 

\vskip5mm

\begin{Thm} \label{main} Let $Y$ and $X$ be two Banach spaces and $T: Y \longrightarrow X$ be a continuous linear operator. Then, $T$ is a limited operator if and only if, for every convex continuous function $f: X \longrightarrow \R$ and every $y\in Y$, the function $f\circ T$ is Fr\'echet differentiable at $y\in Y$ whenever $f$ is G\^ateaux differentiable at $T(y)\in X$.
\end{Thm} 
As consequence we give, in Theorem \ref{BBF} below, new characterizations of infinite dimensional Banach spaces, complementing a result of Borwein and Fabian in [\cite{BF}, Theorem 1.].

\vskip5mm 
A real valued function $f$ on a Banach space will be called a PGNF-function (see \cite{BF}) if there exists a point at
which $f$ is G\^ateaux but not Fr\'echet differentiable. A JN-sequence (due to Josefson-Nissenzweig theorem, see [\cite{Dies}, Chapter XII]) is a sequence $(p_n)_n$ in a dual space $Y^*$ that is weak$^*$ null and $\inf_n \|p_n\|>0$. We say that a function $g$ on $X^*$ has a norm-strong minimum (resp. weak$^*$-strong minimum) at $p\in X^*$ if $g(p)= \inf_{q\in X^*}g(q)$ and $(p_n)_n$ norm converges (resp. weak$^*$ converges) to $p$ whenever $g(p_n)\longrightarrow g(p)$. A norm-strong minimum and weak$^*$-strong minimum are in particular unique. 

\begin{Thm} \label{BBF} Let $Y$ be a Banach space. Then the following assertions are equivalent.

$(1)$ $Y$ is infinite dimensional.

$(2)$ There exists a JN-sequence in $Y^*$.

$(3)$ There exists a convex norm separable and weak$^*$ compact metrizable subset $K$ of $Y^*$ contaning $0$ and a continuous seminorm $h$ on $X^*$ which is weak$^*$ lower semicontinuous and weak$^*$ sequentially continuous, such that the restriction $h_{|K}$ has a weak$^*$-strong minimum but not norm-strong minimum at $0$.


$(4)$ There exists a Banach space $X$ and a linear continuous non-limited operator $T: Y\longrightarrow X$.


$(5)$ There exists on $Y$ a convex continuous PGNF-function.


\end{Thm}

\vskip5mm

In Section \ref{S2} we give some preliminary results, specially the key Lemma \ref{weak}. In Section \ref{S3}, we give the proof of Theorem \ref{main} (divided into two part, Theorem \ref{recip} and  Theorem \ref{princ}) and the proof of Theorem \ref{BBF}. In Section \ref{S4} we give some complementary remarks.
\begin{center}
\section{Preliminaries.} \label{S2}
\end{center}

We recall the following classical result.

\begin{Lem} \label{wellknow} Let $\mathcal{T}_1$ and $\mathcal{T}_2$ be topologies on a set $K$ such that

$(1)$ $K$ is Hausdorff with respect $\mathcal{T}_1$,

$(2)$ $K$ is compact with respect to $\mathcal{T}_2$,

$(3)$ $\mathcal{T}_1 \subset \mathcal{T}_2$.

Then $\mathcal{T}_1 = \mathcal{T}_2$.
\end{Lem}

\begin{proof} Let $F\subset K$ be a $\mathcal{T}_2$-closed set. It follows that $F$ is $\mathcal{T}_2$-compact, since $K$ is $\mathcal{T}_2$-compact. Let $\lbrace \mathcal{O}_i: i\in I\rbrace$ be any cover of $F$ by $\mathcal{T}_1$-open sets. Since $\mathcal{T}_1 \subset \mathcal{T}_2$, then each of these sets is also $\mathcal{T}_2$-open. Hence, there exist
a finite subcollection that covers $F$. It follows that $F$ is $\mathcal{T}_1$-compact and therefore
is $\mathcal{T}_1$-closed since $\mathcal{T}_1$ is Hausdorff. This implies that $\mathcal{T}_2 \subset \mathcal{T}_1$. Consequently, $\mathcal{T}_1 = \mathcal{T}_2$.

\end{proof}

Now, we establish the following useful lemma. If $B$ is a subset of a dual Banach space $X^*$, we denote by $\overline{co}^{w^*}{(B)}$ the weak$^*$ closed convex hull of $B$.

\begin{Lem} \label{weak} Let $X$ be a Banach space and $K$ be a subset of $X^*$.

$(1)$ Suppose that $K$ is norm separable, then there exists a sequence $(x_n)_n$ in the unit sphere $S_X$ of $X$ which separate the points of $K$ i.e. for all $p, p'\in K$, if $\langle p, x_n\rangle = \langle p', x_n\rangle$ for all $n\in \N$, then $p=p'$. Consequently, if $K$ is a weak$^*$ compact and norm separable  set of $X^*$, then the weak$^*$ topology of $X^*$ restricted to $K$ is metrizable.

$(2)$ Let $(p_n)_n$ be a weak$^*$ null sequence in $X^*$. Then, the set $\overline{co}^{w^*}\lbrace p_n: n\in \N\rbrace$ is convex weak$^*$ compact and norm separable.
\end{Lem}
\begin{proof} $(1)$ Since $K$ is norm separable, then $K- K:=\lbrace a- b/ (a, b) \in K\times K \rbrace$ is also norm separable and so there exists a sequence $(q_n)_n$ of $K- K$ which is dense in $K- K$.
According to the Bishop-Phelps  theorem \cite{BP}, the set 
$$D=\{r \in X^*\mid r \text{ attains its supremum on the sphere} \hspace{2mm}S_X\}$$
is norm-dense in the dual $X^*$. Thus, for each $n \in \N$, there exists $r_n \in D$ such that $\|q_n-r_n\|<\frac 1 {1+ n}$. For each 
$n\in \N$, let $x_n \in S_X$ be such that $\|r_n\|=\langle r_n , x_n \rangle $. We claim that the sequence $(x_n)_n$ separate the points of $K$. Indeed, let $q\in K-K$ and suppose that 
$\langle q, x_n \rangle=0$, for all $n\in \N$. There exists a subsequence $(q_{n_k})_k\subset K - K$ such that $\|q_{n_k} - q\|< \frac 1 k$ for all $k\in \N^*$ and so we have $\|r_{n_k} - q\| < \frac 1 {1+n_k} + \frac 1 k$. It follows that 
\begin{eqnarray*} 
\|r_{n_k}\| &= & \langle r_{n_k}, x_{n_k}\rangle \\
                    &=& \langle r_{n_k}, x_{n_k}\rangle - \langle q, x_{n_k}\rangle \\
                    &\leq& \|r_{n_k} - q\|\\
                    &<& \frac 1 {1+ n_k } + \frac 1 k .
\end{eqnarray*}
Hence, for all $k\in \N^*$, $\|q\| \leq \|q-r_{n_k}\| + \|r_{n_k}\| < 2(\frac 1 {1+n_k} + \frac 1 k)$, which implies that $q=0$, and so that $(x_n)_n$ separate the points of $K$. Now, suppose that $K$ is weak$^*$ compact subset of $X^*$. We show that the weak$^*$ topology of $X^*$ restricted to $K$ is metrizable. Indeed, each $x \in X$ determines a seminorm $\nu_x$ on $X^*$ given by $$\nu_x(p)=|\langle p, x \rangle|, \hspace{4mm} p\in X^*.$$
The family of seminorms $(\nu_x)_{x\in X}$ induces the weak$^*$ topology $\sigma(X^*,X)$
on $X^*$. The subfamily $(\nu_{x_n})_n$ also induces a topology on $X^*$, which we will
call $\mathcal{T}$. Since this is a smaller family of seminorms, we have $\mathcal{T} \subseteq \sigma(X^*,X)$. Suppose that $p, p' \in K$ and $\nu_{x_n}(p-p')=0$ for all $n \in \N$. Then we have $\langle p, x_n \rangle = \langle p', x_n \rangle$ for all $n\in \N$ and so we have that $p=p'$ since $(x_n)_n$ separates the points of $K$. Consequently, $K$ is Hausdorff with respect to the topology $\mathcal{T}_{|K}$ (the restriction of $\mathcal{T}$ to $K$). Thus $\mathcal{T}_{|K}$ is a Hausdorff topology on $K$ induced from a countable family
of seminorms, so this topology is metrizable. More precisely, $\mathcal{T}_{|K}$ is induced from the metric
 $$d(p,p'):=\sum_{n=0}^{+\infty} 2^{-n} \frac{\nu_{x_n}(p-p')}{1+\nu_{x_n}(p-p')}.$$
Then we have that $K$ is Hausdorff with respect to $\mathcal{T}_{|K}$, and is
compact with respect to $\sigma(X^*,X)_{|K}$. Lemma \ref{wellknow} implies that
$\mathcal{T}_{|K}=\sigma(X^*,X)_{|K}$. Hence $\sigma(X^*,X)_{|K}$ is metrizable.

$(2)$ Let $(p_n)_n$ be a weak$^*$ null sequence in $X^*$ and set $K=\overline{co}^{w^*}\lbrace p_n: n\in \N\rbrace$. Clearly $K$ is a convex and weak$^*$ compact subset of $X^*$. According to Haydon's theorem [\cite{H}, Theorem 3.3] the weak$^*$ compact convex set $K$ is the norm closed convex hull of its
extreme points whenever $ex(K)$ (the set of extreme points of $K$) is norm separable.  By the Milman theorem  [\cite{Ph1}, p.9]
$ex(K) \subset \overline{\lbrace p_n: n\in \N \rbrace}^{w^*}=\lbrace p_n : n\in \N\rbrace \cup \lbrace 0 \rbrace $ so that $ex(K)$ is norm separable and, hence, by Haydon's theorem, $K$ itself is weak$^*$ compact, convex, and norm separable.
\end{proof}

The following proposition will be used in the proof of Theorem \ref{recip}.

\begin{Prop} \label{gmin} Let $X$ be a Banach space and $K$ be a weak$^*$ compact and norm separable subset of $X^*$ containing $0$.
Then, there exists a continuous seminorm $h$ on $X^*$ satisfying

$(1)$ $h$ is weak$^*$ lower semicontinuous and sequentially weak$^*$ continuous,  

$(2)$ the restriction $h_{|K}$ of $h$ to $K$ has a weak$^*$-strong minimum at $0$.


\end{Prop}

\begin{proof} Using Lemma \ref{weak}, there exists a sequence $(x_k)_k\subset S_X$ which separate the points of $K$. Define the function $h : X^* \longrightarrow \R $ as follows:
$$h(x^*)=(\sum_{k\geq 0} 2^{-k}(\langle x^*, x_k\rangle)^2)^{\frac 1 2}, \hspace{4mm} \forall x^*\in X^*.$$
It is clear that $h$ is a seminorm, and since $h(x^*) \leq \|x^*\|$ for all $x^*\in X^*$, it is also continuous. Since $h$ is the supremum of a sequence of weak$^*$ continuous functions, it is weak$^*$ lower semicontinuous.
On the other hand, since the series $\sum_{k\geq 0} 2^{-k}(\langle x^*, x_k\rangle)^2$ uniformly converges on bounded sets of $X^*$ and since the maps $\hat{x}_k: x^* \mapsto \langle x^*, x_k\rangle$ are weak$^*$ continuous for all $k\in \N$, then $h$ is sequentially weak$^*$ continuous. If $p\in K$ and $h(p)=0$, then $\langle p, x_k\rangle=0$ for all $k\in \N$ which implies that $p=0$, since the sequence $(x_k)_k$ separate the points of $K$. Hence, the restriction of $h$ to $K$ has a unique minimum at $0$. This minimum is necessarily a weak$^*$-strong minimum since $K$ is weak$^*$ metrizable by Lemma \ref{weak}, this follows from a general fact which say that for every lower semicontinuous function on a compact metric space $(K,d)$, a unique minimum is necessarily a strong minimum for the metric $d$ in question. 
\end{proof}


\begin{center}
\section{Limited operators and differentiability.} \label{S3}
\end{center}

Recall that the domain of a function $f: X \longrightarrow \R\cup \lbrace+\infty\rbrace$, is the set 
$$dom(f):=\lbrace x \in X/ f(x) < +\infty\rbrace.$$ 
For a function $f$ with $dom(f)\neq \emptyset$, the Fenchel transform of $f$ is defined on the dual space for all $p\in X^*$ by
$$f^*(p):=\sup_{x\in X} \lbrace \langle p,x \rangle - f(x)\rbrace.$$
The second transform $(f^*)^*$ is defined on the bidual $X^{**}$ by the same formula. We denote by $f^{**}$, the restriction of $(f^*)^*$ to $X$, where $X$ is identified to a subspace of $X^{**}$ by the canonical embedding. Recall that the Fenchel theorem state that $f=f^{**}$ if and only if $f$ is convex lower semicontinuous on $X$. 

\vskip5mm

The "if" part of Theorem \ref{main} is given by the following theorem.

\vskip5mm

\begin{Thm} \label{recip} Let $Y$ and $X$ be Banach spaces and let $T: Y \longrightarrow X$ be a linear continuous operator. Suppose that $f\circ T$ is Fr\'echet differentiable at $y\in Y$ whenever $f: X \longrightarrow \R $ is convex continuous and  G\^ateaux differentiable at $T(y)\in Y$. Then $T$ is a limited operator.
\end{Thm}

\begin{proof} Let $(p_n)_n$ be a weak$^*$ null sequence in $X^*$. We want to prove that $\|T^*(p_n)\|_{Y^*}\rightarrow 0$.  
Set $$K=\overline{co}^{w^*}\lbrace p_n: n\in \N\rbrace.$$ According to Lemma \ref{weak}, $K$ is convex weak$^*$ compact and norm separable. Using Proposition \ref{gmin}, there exists a continuous seminorm which is weak$^*$ lower semicontinuous and sequentially weak$^*$ continuous $h : X^* \longrightarrow \R $ such that the restriction $h_{|K}$ of $h$ to $K$ has a weak$^*$-strong minimum at $0$ and in particular $\min_K h=h(0)=0$. Since the sequence $(p_n)_n$ weak$^*$ converges to $0$, it follows that $\lim_{n} h(p_n)=h(0)=\min_K h$. Thus, $(p_n)_n$ is a minimizing sequence for $h_{|K}$. Set $g= h+\delta_{K}$, where $\delta_{K}$ denotes the indicator function, which is equal to $0$ on $K$ and equal to $+\infty$ otherwise. Since $K$ is convex, weak$^*$-closed and norm bounded, then $g$ is a convex and weak$^*$ lower semicontinuous function with a norm bounded domain $dom(g)=K$. Moreover we have, 

$(1)$ $g(p) > 0=g(0)=\min_{X^*} (g)$ for all $p\in X^*\setminus \lbrace 0 \rbrace$.

$(2)$ $\lim_{n\rightarrow +\infty} g(p_n)=\min_{X^*} (g)$.
\vskip5mm
Hence, there exists a convex and Lipschitz continuous function $f : X \longrightarrow \R$ such that $g=f^*$ (we can take $f=g^*_{|X}$). The function $f$ is  G\^ateaux differentiable at $0$ with G\^ateaux derivative $\nabla f(0)=0$, this is due to the fact that $f^*=g$ has a weak$^*$-strong minimum at $0$ (we can see [Corollary 1. \cite{AR}]). Thus, from our hypothesis, $f\circ T$ is Fr\'echet differentiable at $0$ with Fr\'echet derivative equal to $0$. It follows that $(f\circ T)^*$ has a norm-strong minimum at $0$ (see [Corollary 2. \cite{AR}]). Now, we prove that $(T^*(p_n))_n$ is a minimizing sequence for $(f\circ T)^*$, which will implies that $\|T^*(p_n)\|_{Y^*}\rightarrow 0$. Indeed, on one hand, we have  $0=\min_{X^*} (g)=-g^*(0)=-f(0)$. On the other hand we have
\begin{eqnarray*}
0=-f(0) \leq  \sup_{y\in Y}\lbrace  - f\circ T(y)\rbrace&:=& (f\circ T)^*(0)\\
                           &\leq& \sup_{x\in X}\lbrace - f(x)\rbrace\\
                      &=& f^*(0)\\
                      &=& g(0)\\
                      &=& 0.
\end{eqnarray*}
It follows that $(f\circ T)^*(0)=0$. Hence, since $(f\circ T)^*$ has a minimum at $0$, we obtain
\begin{eqnarray*}
0=(f\circ T)^*(0) \leq (f\circ T)^*(T^*(p_n))&:=&\sup_{y\in Y}\lbrace \langle T^*(p_n), y \rangle - f\circ T(y)\rbrace \\
                      &=& \sup_{y\in Y}\lbrace \langle p_n, T(y) \rangle - f(T(y))\rbrace\\
                      &\leq& \sup_{x\in X}\lbrace \langle p_n, x \rangle - f(x)\rbrace\\
                      &=& f^*(p_n)\\
                      &=& g(p_n).
\end{eqnarray*}
Since, $g(p_n)\rightarrow 0$, it follows that $(f\circ T)^*(T^*(p_n))\rightarrow 0=(f\circ T)^*(0)$. In other words, $(T^*(p_n))_n$ is a minimizing sequence for $(f\circ T)^*$. Since $(f\circ T)^*$ has a norm-strong minimum at $0$, we obtain that $\|T^*(p_n)\|_{Y^*}\rightarrow 0$, which implies that $T$ is a limited operator. 
\end{proof}

\vskip5mm

The "only if" part of Theorem \ref{main} is given by the following theorem. 

\begin{Thm} \label{princ}  Let $Y$ and $X$ be two Banach spaces and $T: Y \longrightarrow X$ be a limited operator. Let $f: X \longrightarrow \R \cup \lbrace +\infty \rbrace $, be a convex lower semicontinuous function and let $a \in Y$ such that $T(a)$ belongs to the interior of $dom(f)$. Then, $f\circ T$ is Fr\'echet differentiable at $a\in Y$ with Fr\'echet-derivative $T^*(Q)\in Y^*$, whenever $f$ is G\^ateaux differentiable at $T(a)\in X$ with G\^ateaux-derivative $Q\in X^*$.
\end{Thm}

\begin{proof}
Since $f$ is convexe lower semicontinuous and $T(a)$ is in the interior of $dom(f)$, there exists $r_a>0$ and $L_a>0$ such that $f$ is $L_a$-Lipschitz continuous on the closed ball $B_X(T(a),r_y)$. It is well known that there exists a convex $L_a$-Lipschitz continuous function $\tilde{f}_a$ on $X$ such that $\tilde{f}_a=f$ on $B_X(T(a),r_a)$ (See for instance Lemma 2.31 \cite{Ph}). It follows that $\tilde{f}_a \circ T = f\circ T$ on $B_Y(a,\frac{r_a}{\|T\|})$, since $T(B_X(a,\frac{r_a}{\|T\|}))$ is a subset of $B_X(T(a),r_a)$ (we can assume that $T\neq 0$).  Replacing $f$ by $\frac{1}{L_a}\tilde{f}_a$, we can assume without loss of generality that $f$ is convexe $1$-Lipschitz continuous on $X$. It follows that $dom(f^*)\subset B_{X^*}$ (the closed unit ball of $X^*$). 

\vskip5mm

{\it Claim.} Suppose that $f$ is G\^ateaux differentiable at $T(a)\in X$ with G\^ateaux-derivative $Q\in X^*$, then the function $q\mapsto f^*(q)-\langle q, T(a)\rangle$ has a weak$^*$-strong minimum on $B_{X^*}$ at $Q$.
\vskip5mm
{\it Proof of the claim.} See [Corollary 1. \cite{AR}].

\vskip5mm

Now, suppose by contradiction that $T^*(Q)$ is not the the Fr\'echet derivative of $f\circ T$ at $a$. There exist $\varepsilon >0$, $t_n\longrightarrow 0^+$ and $h_n\in Y$, $\|h_n\|_Y=1$ such that for all $n\in \N^*$, 

\begin{eqnarray} \label{eq4}
f\circ T(a+t_n h_n)- f\circ T(a) - \langle T^*(Q), t_n h_n \rangle > \varepsilon t_n.
\end{eqnarray}
Let $r_n=t_n/n$ for all $n\in \N^*$ and choose $p_n\in B_{X^*}$ such that
\begin{eqnarray} \label{eq5}
f^*(p_n)  - \langle p_n, T(a+t_n h_n)\rangle < \inf_{p\in B_{X^*}}\lbrace f^* (p) - \langle  p, T(a+ t_n h_n)\rangle \rbrace + r_n.
\end{eqnarray}
 From (\ref{eq5}) we get 
\begin{eqnarray*} 
f^*(p_n) - \langle p_n, T(a) \rangle  &< & \inf_{p\in B_{X^*}}\lbrace f^* (p) - \langle  p, T(a)\rangle \rbrace + 2 t_n \|T\|+ r_n.
\end{eqnarray*} 
This implies that the sequence $(p_n)_n$ minimize the function $q\mapsto f^*(q)-\langle q, T(a)\rangle$ on $B_{X^*}$. Using the claim, the function $q\mapsto f^*(q)-\langle q, T(a)\rangle$ has a weak$^*$-strong minimum on $B_{X^*}$ at $Q$, it follows that $(p_n)_n$ weak$^*$ converges to $Q$ and so (since $T$ is limited) we have 
\begin{eqnarray} \label{eq5'}
\|T^*(p_n -Q)\|_{Y^*}\longrightarrow 0.
\end{eqnarray}

\noindent On the other hand, since $f(T(a+t_n h_n))=f^{**}(T(a+t_n h_n))=- \inf_{p\in B_{X^*}}\lbrace f^* (p) - \langle  p, T(a+ t_n h_n)\rangle \rbrace$, using (\ref{eq5}) we obtain for all $y\in Y$

\begin{eqnarray*} 
f\circ T(a+t_n h_n) - \langle p_n, T(a+t_n h_n)\rangle &<& - f^*(p_n)  + r_n \nonumber \\
                                             &\leq& f\circ T(y) - \langle p_n, T(y)\rangle + r_n.
\end{eqnarray*}
Replacing $y$ by $a$ in the above inequality we obtain
\begin{eqnarray} \label{eq6}
f\circ T(a+t_n h_n) - \langle p_n, T(t_n h_n)\rangle &\leq& f\circ T(a) + r_n.
\end{eqnarray}
Combining (\ref{eq4}) and (\ref{eq6}) we get 
\begin{eqnarray} \label{eq7}
\varepsilon &<& \langle p_n, T(h_n)\rangle - \langle T^*(Q),h_n \rangle + r_n/t_n \nonumber\\
            &=& \langle T^*(p_n), h_n\rangle - \langle T^*(Q), h_n \rangle +\frac{1}{n} \nonumber\\
            &\leq& \|T^*(p_n -Q)\|_{Y^*} +\frac{1}{n} \nonumber        
\end{eqnarray}
which is a contradiction with (\ref{eq5'}). Thus $f\circ T$ is Fr\'echet differentiable at $a$ with Fr\'echet derivative $T^*(Q)$.
\end{proof}
\vskip5mm
Now, we give the proof of Theorem \ref{BBF}.

\begin{proof}[\bf Proof of Theorem \ref{BBF}] $(1)\Longrightarrow (2)$ is the deeper Josefson-Nissenzweig theorem [\cite{Dies}, Chapter XII]. 

$(2)\Longrightarrow (1)$ is well known.
 
$(2)\Longrightarrow  (3)$ Let $(p_n)_n$ be a weak$^*$ null sequence in $Y^*$ such that $\inf_n \|p_n\|>0$ and set $K=\overline{co}^{w^*}\lbrace p_n: n\in \N\rbrace$. By Lemma \ref{weak}, the set $K$ is convex norm separable and weak$^*$ compact metrizable. On the other hand, from Proposition \ref{gmin}, there exists a continuous seminorm $h$ which is weak$^*$ lower semicontinuous  and weak$^*$ sequentially continuous on $Y^*$ such that the restriction of $h$ to $K$ has a weak$^*$-strong minimum at $0$. It remains to show that $0$ is not a norm-strong minimum for $h_{|K}$. Indeed, since $(p_n)_n$ is weak$^*$ null and $h$ is weak$^*$ sequentially continuous, then $\lim_n h(p_n)=h(0)=\min_K h$. So $(p_n)_n$ is a minimizing sequence for $h_{|K}$ which not converges to $0$ since $\inf_n \|p_n\|>0$. Hence, $0$ is not a norm-strong minimum for $h_{|K}$.

$(3)\Longrightarrow  (2)$ Since $0$ is not a norm-strong minimum for the restriction $h_{|K}$, there exists a sequence $(p_n)_n$ that minimize $h$ on $K$ but $\|p_n\|\nrightarrow 0$. Since $h_{|K}$ has a weak$^*$-strong minimum at $0$, it follows that $(p_n)_n$ weak$^*$ converges to $0$. Hence, $(p_n)_n$ weak$^*$ converges to $0$ but $\|p_n\|\nrightarrow 0$. Thus, there exists a JN-sequence in $Y^*$. 

$(2) \Longrightarrow (4)$ This part is given by taking $X=Y$ and $T=I$ the identity map. Indeed, there exists a sequence $(p_n)_n$ which weak$^*$ converges to $0$ but $\inf_n \|I^*(p_n)\| =\inf_n \|p_n\|>0$. So $I$ cannot be a limited operator. 

$(4)\Longrightarrow (5)$. Indeed, if there exists a Banach space $X$ and a non-limited operator $T: Y \longrightarrow X$, by using Theorem \ref{main}, there exists a convex continuous function $f: X\longrightarrow \R$ and a point $y\in Y$ such that $f$ is  G\^ateaux differentiable at $T(y)\in X$ but $f\circ T$ is not Fr\'echet differentiable at $y$. So $f\circ T$ is G\^ateaux but not Fr\'echet differentiable at $y$. Hence, $f\circ T$ is a convex continuous PGNF-function on $Y$. 

$(5) \Longrightarrow (2)$ Let $f$ be a PGNF-function on $Y$. We can assume without loss of generality that $f$ is G\^ateaux differentiable at $0$ with G\^ateaux-derivative equal to $0$, but $f$ is not Fr\'echet differentiable at $0$. It follows from classical duality result (see Corollary 1. and Corollary 2. in \cite{AR}) that $f^*$ has a weak$^*$-strong minimum but not norm-strong minimum at $0$. Since $0$ is not a norm-strong minimum for $f^*$, there exists a sequence $(p_n)_n\in X^*$ minimizing $f^*$ such that $\|p_n\|\nrightarrow 0$. On the other hand, since $f^*$ has a weak$^*$-strong minimum at $0$, and $(p_n)_n$ minimize $f^*$, we have that $(p_n)_n$ weak$^*$ converges to $0$. Thus, $(p_n)_n$ weak$^*$ converges to $0$ but $\|p_n\|\nrightarrow 0$. Hence, there exists a JN-sequence.
\end{proof}

\paragraph{Canonical construction of PGNF-function.} There exist different way to build a PGNF-function in infinite dimentional Banach spaces. We can find examples of such constructions in \cite{BF}. We present below a different method for constructing a PGNF-function on a Banach space $X$ canonically from a JN-sequence. Given a JN-sequence $(p_n)_n\subset X^*$, we set $K=\overline{co}^{w^*}\lbrace p_n: n\in \N\rbrace$. Using Lemma \ref{weak}, there exists a sequence $(x_n)_n \in S_X$ which separates the points of $K$, and as in the proof of Proposition \ref{gmin}, there exist a continuous seminorm $h$ which is weak$^*$ lower semicontinuous and weak$^*$ sequentially continuous such that $h_{|K}$ has a weak$^*$-strong minimum at $0$. The function $h$ is given explicitly as follows 
$$h(x^*)=(\sum_{n\geq 0} 2^{-n}(\langle x^*, x_n\rangle)^2)^{\frac 1 2}, \hspace{4mm} \forall x^*\in X^*.$$
Since $(p_n)_n$ weak$^*$ converges to $0$, it follows that $(p_n)$ is a minimizing sequence for $h_{|K}$. Since $(p_n)_n$ is a JN-sequence, it follows that $0$ is not a norm-strong minimum for $h_{|K}$. Define the function $f$ by
$$f(x)= (h + \delta_K)^*(\hat{x}), \hspace{4mm} \forall x\in X,$$ where $\delta_{K}$ denotes the indicator function, which is equal to $0$ on $K$ and equal to $+\infty$ otherwise and where for each $x\in X$, we denote by  $\hat{x}\in X^{**}$ the linear map $x^* \mapsto \langle x^*, x \rangle$ for all $x^*\in X^*$. Then $f$ is convex Lipschitz continuous, G\^ateaux differentiable at $0$ (since $h + \delta_K$ has a weak$^*$-strong minimum) but not Fr\'echet differentiable at $0$ (because $0$ is not a norm-strong minimum for $h + \delta_K$).

\begin{center}
\section{Appendix.} \label{S4}
\end{center}

There exists a class of Banach spaces $(E,\|.\|_E)$ such that the canonical embedding $i: E\longrightarrow E^{**}$ is a limited operator. This class contains in particular the space $c_0$ and any closed subspace $F$ of $c_0$ (This class is also stable by product and quotient. For more information see \cite{CGL}). In this setting, Theorem \ref{princ} gives immediately the following corollary.
\vskip5mm
\begin{Cor} \label{Za1} Suppose that the canonical embedding $i: E\longrightarrow E^{**}$ is a limited operator. Let $g: E^{**} \longrightarrow \R \cup \lbrace +\infty \rbrace$ be a convex lower semicontinuous function. Suppose that $x\in E$ belongs to the interior of  $dom(g)$ and that $g$ is G\^ateaux differentiable at $x\in E$ (we use the identification $i(x)=x$), then the restriction of $g$ to $E$ is Fr\'echet differentiable at $x$. In particular, if $f : E \longrightarrow \R \cup \lbrace +\infty \rbrace$ is convex lower semicontinuous function, $x\in E$ belongs to the interior of $dom((f^{*})^*)$ and $(f^{*})^*$ is G\^ateaux differentiable at $x$, then $f$ is  Fr\'echet differentiable at $x$. 
\end{Cor}

We obtain the following corollary by combining Proposition \ref{Za2} and a delicate result due to Zajicek (see [Theorem 2; \cite{Za}]), which say that in a separable Banach space, the set of the points where a convex continuous function is not G\^ateaux differentiable, can be covered by countably many $d.c$ (that is, delta-convex) $hypersurface$. Recall that in a separable Banach space $Y$, each set $A$ which can be covered by countably many $d.c$ $hypersurface$ is $\sigma$-lower porous, also $\sigma$-directionally porous; in particular it is both Aronszajn (equivalent to Gauss) null and $\Gamma$-null. For details about this notions  of small sets we refer to \cite{Za1} and references therein. Note that a limited set in a separable Banach space is relatively compact \cite{BD}.

\begin{Prop} \label{Za2} Let $Y$ and $X$ be Banach spaces and $T: Y \longrightarrow X$ be a limited operator with a dense range. Let $f : X \longrightarrow \R$ be a convex continuous function. Then $f\circ T$ is G\^ateaux differentiable at $a\in Y$ if and only if, $f\circ T$ is Fr\'echet differentiable at $a\in Y$. 
\end{Prop}
\begin{proof} Suppose that $f\circ T$ is G\^ateaux differentiable at $a\in Y$. It follows that $f$ is G\^ateaux differentiable at $T(a)$ with respect to the direction $T(Y)$ which is dense in $X$. It follows (from a classical fact on locally Lipschitz continuous functions) that $f$ is G\^ateaux differentiable at $T(a)$ on $X$. So by Theorem \ref{princ}, $f\circ T$ is Fr\'echet differentiable at $a\in Y$. The converse is always true.

\end{proof}

\begin{Cor} \label{Za}  Let $Y$ be a separable Banach space, $X$ be a Banach spaces and $T: Y \longrightarrow X$ be a compact operator with a dense range. Let $f: X \longrightarrow \R$, be a convex and continuous function. Then, the set of all points at which $f\circ T$ is not Fr\'echet differentiable can be covered by countably many $d.c$ $hypersurface$.
\end{Cor}

\bibliographystyle{amsplain}

\end{document}